\documentclass[lettersize,journal]{IEEEtran}
\usepackage{amsmath,amsfonts}
\usepackage{algorithmic}
\usepackage{algorithm}
\usepackage{array}
\usepackage{textcomp}
\usepackage{stfloats}
\usepackage{url}
\usepackage{verbatim}
\usepackage{graphicx}
\usepackage{cite}
\usepackage{makecell}
\usepackage{multirow}
\usepackage{bm}
\usepackage{cancel}
\usepackage{subcaption}
\usepackage[hidelinks]{hyperref}
\usepackage{amsthm}
\usepackage{makecell}
\theoremstyle{definition}

\newtheorem{definition}{Definition}
\newtheorem{remark}{Remark}
\newtheorem{proposition}{Proposition}
\usepackage[some]{background}
\SetBgScale{1}
\SetBgContents{\parbox{10cm}{%
  \Huge Draft:  \today\\[18cm]\rotatebox{180}{\Huge Draft:  \today}}}
\SetBgColor{red}
\SetBgAngle{270}
\SetBgOpacity{0.2}

\newcommand{\notsure}[1]{{\color{black} #1}}

\begin{document}


\title{Thermodynamical Material Networks for Modeling, Planning, and Control of Circular Material Flows}

\author{Federico Zocco, Pantelis Sopasakis, Beatrice Smyth, and Wassim M. Haddad
\thanks{F. Zocco is with the Centre for Intelligent Autonomous Manufacturing Systems, School of Electronics, Electrical Engineering and Computer Science, Queen's University Belfast and with the Research Centre in Sustainable Energy, School of Mechanical and Aerospace Engineering, Queen's University Belfast, Northern Ireland, UK. Email: federico.zocco.fz@gmail.com}
\thanks{Corresponding authors: \emph{Federico Zocco}}
\thanks{P. Sopasakis is with the Centre for Intelligent Autonomous Manufacturing Systems, School of Electronics, Electrical Engineering and Computer Science, Queen's University Belfast, Northern Ireland, UK. Email: p.sopasakis@qub.ac.uk.}
\thanks{B. Smyth is with the Research Centre in Sustainable Energy, School of Mechanical and Aerospace Engineering, Queen's University Belfast, Northern Ireland, UK. Email: beatrice.smyth@qub.ac.uk.}
\thanks{W. M. Haddad is with the School of Aerospace Engineering, Georgia Institute of Technology, Atlanta, GA, USA. Email: wm.haddad@aerospace.gatech.edu.}
}

%

\maketitle

\begin{abstract}
Waste production, carbon dioxide atmospheric accumulation, and dependence on finite natural resources are expressions of the unsustainability of the current industrial networks that supply fuels, energy, and manufacturing products. In particular, circular manufacturing supply chains and carbon control networks are urgently needed. To model and design these and, in general, any material networks, we propose to generalize the approach used for traditional networks such as water and thermal power systems by using compartmental dynamical thermodynamics and graph theory. The key idea is that the thermodynamic compartments and their connections can be added, removed or modified as needed to achieve a circular material flow. The design methodology is explained and its application is illustrated through examples. In addition, we provide a physics-based definition of circularity and, by implementing a nonlinear compartmental control, we strengthen the connection between ``Industry 4.0'' and ``Sustainability''. The paper source code is publicly available\footnotemark{}. \footnotetext{https://github.com/fedezocco/TMNbiometh-SciPy}   
\end{abstract}


\section{Introduction}
While the human population is predicted to reach 8 billion by 2024 and 10 billion by 2056 \cite{owidworldpopulationgrowth}, modern society strives to provide the needed services and products on a large-scale. Any services and products, from the simple piece of paper to the complex graphical processing unit, require the availability of raw materials and the production of energy. As welfare and economic growth rely on material uses, a long-term sustainable management of finite natural resources is increasing in importance \cite{EllenMacArthurOnCE}. 

A natural resource particularly at risk is climate stability, which is mostly being altered by the atmospheric carbon dioxide concentration \cite{NASAclimateChange}. The world emitted 6 billion tonnes of $\text{CO}_2$ in 1950, 22 billion tonnes in 1990, 36 billion tonnes in 2019 (i.e., 6 times the emissions of 1950) and the annual emissions have yet to reach their peak \cite{owidco2andothergreenhousegasemissions}. Therefore, an effective control of this material is necessary to respect the global warming target of 1.5 degrees Celsius compared to the pre-industrial levels established with the Paris Agreement in 2015 \cite{ParisAgreement-UN}.

Along with carbon dioxide, other materials requiring a more efficient management are those accumulating on lands and seas as litter or marine debris such as plastic. For example, the mass of plastic in the Great Pacific Garbage Patch was estimated to be approximately 80,000 tonnes and the mass of plastic entering the ocean each year is 1.15 to 2.41 million metric tonnes \cite{OceanCleanup-GPGP}. In terms of the life-cycle of a material, the status of ``waste'' is at the final stage, hence waste accumulations are issues related to the end of the life-cycle. Similarly, today there are also increasing concerns at the beginning of the material life-cycle, i.e., at the stage of material extraction. Indeed, there are several materials classified as ``critical'' by the European Union \cite{CRMs-EU} and the United States \cite{CRMs-US} whose supply is particularly at risk. Those materials are irreplaceable in clean technologies such as solar panels, wind turbines, electric vehicles and are also used in modern technologies such as smartphones.

To address the issues located both at the beginning and at the end of the life-cycle of materials, the paradigm of ``circular economy'' has gained attention over the last few years. Currently there are multiple definitions of circular economy \cite{kirchherr2017conceptualizing}. In this paper, we focus on the flows and accumulations of materials (e.g., carbon dioxide, gold, plastic, biomethane), and hence for us the adjective ``circular'' means ``closed flow of material''. As a consequence, the expression ``circular economy'' is equivalent to ``economy based on closed flows of materials'', the expression ``measuring the circularity of a supply network'' is equivalent to ``measuring to what extent the flow of material in a supply network is closed'' and the expression ``circulating a material'' means ``closing the flow of a material''. For example,  hydraulic engineers seek to circulate the water by minimizing the leakages in the water network. 

To enhance the modeling, planning, and control of \emph{circular} material flows, initially we looked at the advanced and mature water industry and asked the question: Given that water is just a particular type of material, can we develop a capability of managing other materials as effective as the one we have with water? Then, we observed two key aspects of the design of water networks. Namely, 
\begin{enumerate}
\item{they are designed to be \emph{closed} in order to minimize leakages of material and a mathematical framework that effectively depicts the network architecture is graph theory \cite{deuerlein2008decomposition}; and}
\item{their modeling is based on the first principle of thermodynamics and the mass conservation equation \cite{kaminski2017introduction}.} 
\end{enumerate}
Given the generality of both thermodynamics \cite{haddad2017thermodynamics} and graph theory \cite{bondy1976graph}, in this paper we propose to extend the modeling approach of water supply networks to model the flow of any material leveraging compartmental dynamical thermodynamics \cite{haddad2019dynamical} and graph theory \cite{bondy1976graph}.  

The main contributions of the paper are the following.
\begin{enumerate}
\item{We provide physics-based foundations of material circularity to add clarity to the topic \cite{kirchherr2017conceptualizing} (see Section \ref{sec:CircAndTMNs}).} 
\item{We propose a systematic methodology to model, plan, and control circular flows of materials (see Section \ref{sec:Method}).} 
\item{We illustrate the use of graphs to measure material circularity (see Section \ref{sec:Examples}).}
\item{We illustrate the use of feedback control systems into the design of material flows (see Section \ref{sec:Examples}). By doing this, we strengthen the link between techniques typical of industrial automation and the holistic perspective of industrial ecology required to design closed-loop flows \cite{bakshi2022sustainability}.}
\end{enumerate}

The paper is organized as follows: Section \ref{sec:RelWork} covers related works, Section \ref{sec:CircAndTMNs} defines the key concepts, Section \ref{sec:Method} details the proposed methodology, Section \ref{sec:Examples} provides two examples, and finally Section \ref{sec:Conclusion} concludes.
 
Throughout the paper, vectors and matrices are indicated with bold letters, whereas sets are indicated with calligraphic letters.

\section{Related Work}\label{sec:RelWork}
\textbf{Compartmental and dynamical thermodynamics:} Compartmental thermodynamics refers to the thermodynamic analysis at \emph{equilibrium} of a set of connected machines. An example of compartmental thermodynamics is the simple Rankine cycle, in which a turbine, a boiler, a water pump, and a condenser are connected through pipes. Its invention dates back to 1859. Dynamical thermodynamics, instead, is an emerging topic as it studies the dynamical (i.e., \emph{non-equilibrium}) behavior of systems from a thermodynamic perspective without a focus on the multi-machine nature typical of compartmental thermodynamics. Examples of works are \cite{freitas2020stochastic} for electrical networks, \cite{freitas2021stochastic} for electronic circuits, \cite{gay2018lagrangian} for mechanical systems, and \cite{avanzini2021nonequilibrium,penocchio2021nonequilibrium} for chemical reaction networks. The combination of these two branches of thermodynamics is compartmental dynamical thermodynamics \cite{haddad2019dynamical}.

\textbf{Graph theory and thermodynamics for sustainability:} In the context of circular economy, some graph-based approaches have been proposed recently. Moktadir \emph{et al.} \cite{moktadir2018drivers} used a graph architecture to examine and prioritize the driving factors of sustainable manufacturing practices; in \cite{singh2020managing} and \cite{how2018debottlenecking} a graph architecture is used to analyze the different barriers to the implementation of a circular economy in the mining industry and in a biomass supply chain, respectively. The work of Gribaudo \emph{et al.} \cite{gribaudo2020circular} proposed the use of graphs to model the production of chitin by bio-conversion of municipal waste.

The idea of using thermodynamics for ecological modeling dates back almost thirty years \cite{schneider1994complexity}. In 2011, the authors of \cite{bakshi2011thermodynamics} further extended this vision by proposing thermodynamics as the science of sustainability. With this paper, we aim at clarifying the application of thermodynamic principles for sustainable and circular design.

\textbf{Material flow analysis:} Material flow analysis (MFA) is one of the key techniques developed and used in industrial ecology and circular economy to assess material flows and stocks in urban and natural environments \cite{graedel2019material}. The holistic perspective at the core of MFA had a strong influence on the methodology proposed in this paper. While MFA is mainly based on mass balances \cite{brunner2016handbook}, our methodology extends it by adding dynamical power balances \cite{haddad2019dynamical} and control systems \cite{haddad2011nonlinear}.

\textbf{Control systems in life-cycles:} Nowadays, control theory is a well-established discipline sitting between applied mathematics and engineering. Two selected works from the vast literature in the field are \cite{doyle2013feedback} for linear feedback control and \cite{haddad2011nonlinear} for advanced non-linear methods. Control systems are distributed across the entire life-cycle of products and services, therefore they can play a critical role in the transition from a linear to a circular economy. In this paper, we illustrate the use of control theory into the design of material flows, specifically for bio-methane production.

\section{Circularity and Thermodynamical Material Networks}\label{sec:CircAndTMNs}
Consider a cube of material $\beta$, infinitesimal mass d$m$, density $\rho$, and volume $\text{d}V = \frac{\text{d}m}{\rho}$ as in Fig. \ref{fig:MechDefinition}. Let $\bm{G}(t) = [x_G(t), y_G(t), z_G(t)]^\top$ be the center of mass whose coordinates are written with respect to a fixed reference frame with origin $\bm{O}= [0, 0 ,0]^\top$. In general, $x_G(t)$, $y_G(t)$, and $z_G(t)$ can vary with the time $t$. Let $\bm{p}(t) = \bm{G}(t) - \bm{O}$ be the position vector of the cube center of mass $\bm{G}$.
\begin{definition}[Mechanics-based circularity]\label{def:MechDef}
 The flow of $\beta$ is \emph{mechanically circular} if there exist $t_0 \geq 0$ and $t^* \in (t_0, \infty)$ such that 
\begin{equation}\label{eq:circularityCond}
\bm{p}(t_0) = \bm{p}(t^*), \quad \dot{\bm{p}}(t_0) \neq 0, \quad  t^* > t_0.
\end{equation}     
\end{definition}
\begin{figure}
    \centering
    \includegraphics[width=0.5\textwidth]{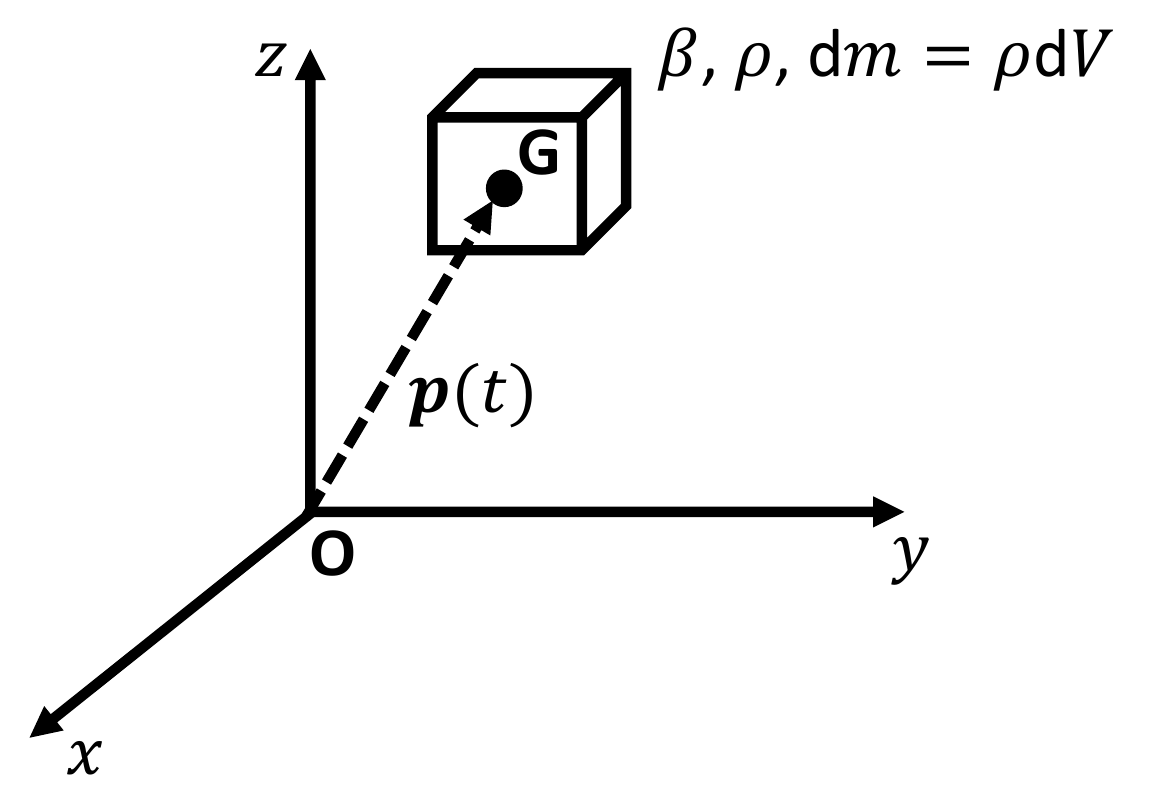}
    \caption{Infinitesimal cube of material $\beta$ and center of mass $\bm{G}$.}
    \label{fig:MechDefinition}
\end{figure}   
\begin{remark}   
As the material $\beta$ in Definition \ref{def:MechDef} is fixed, chemical reactions that modify the material are excluded. Therefore, we refer to  Definition \ref{def:MechDef} as the mechanics-based definition of circularity.
\end{remark}

In thermal engineering, it is standard practice to define a \emph{control volume} that contains the system under study before the application of mass and energy balances. Such a standard practice underpins the design of thermodynamic cycles, e.g., the Rankine and the Brayton cycles, and also the design of hydraulic networks \cite{kaminski2017introduction}. Each control volume identifies a \emph{thermodynamic compartment}. For example, a simple ideal Rankine cycle is made of eight thermodynamic compartments: a feedwater pump, a boiler, a turbine, a condenser, and four pipes connecting these four machines into a closed-loop.  

Now note that mass and energy balances, that is, thermodynamics, are general principles valid for any system \cite{haddad2017thermodynamics,bakshi2011thermodynamics}. Hence, we can generalize the definition of circularity based on mechanics (Definition \ref{def:MechDef}) with the following thermodynamics-based definition.

Let $c^k$ be the $k$-th thermodynamic compartment identified by the $k$-th control volume and let $\beta$ be the material of an infinitesimal cube as in Fig. \ref{fig:MechDefinition}. 
\begin{definition}[Thermodynamics-based circularity]\label{def:ThermoDef}
The flow of $\beta$ is \emph{thermodynamically circular} if there exists an ordered sequence of compartments $\phi = (c^1, \dots, c^k, \dots, c^1)$ processing $\beta$, which begins and ends in $c^1$. Moreover, if some $c^k \in \phi$ chemically transforms a material $\beta_1$ into a material $\beta_2$ and there exists an ordered sequence $\phi$ processing $\mathcal{B} = \{\beta_1, \beta_2\}$, then the flow of the material set $\mathcal{B} = \{\beta_1, \beta_2\}$ is \emph{thermodynamically circular}. More generally, the flow of $\mathcal{B} = \{\beta_1, \dots, \beta_q, \dots, \beta_{n_\beta}\}$ is \emph{thermodynamically circular} if there exists an ordered sequence $\phi$ processing $\mathcal{B}$.            
\end{definition} 
Figure \ref{fig:ThermoDefinition} is an example for $n_{\beta} = 2$. 
\begin{figure}
    \centering
    \includegraphics[width=0.5\textwidth]{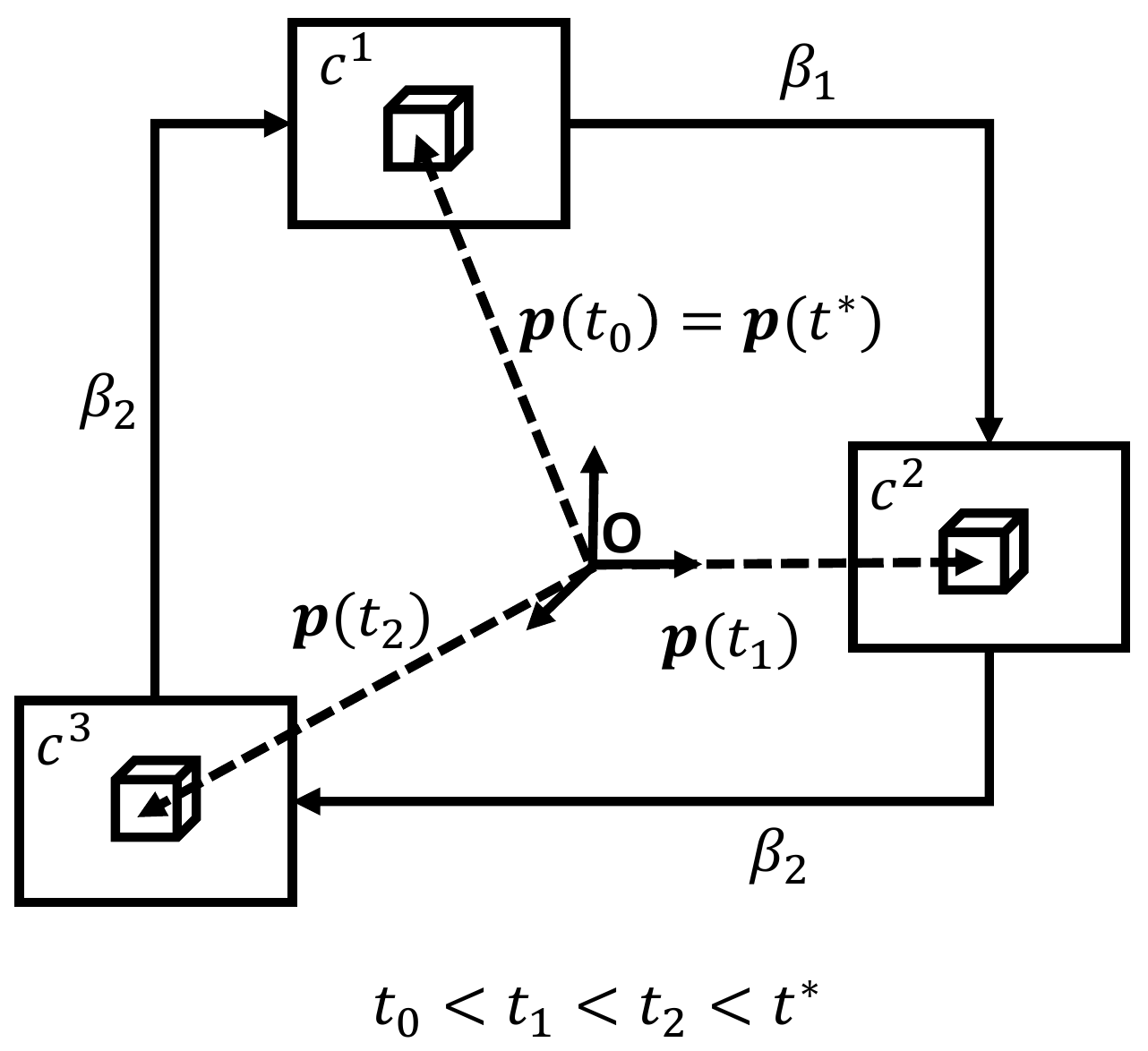}
    \caption{Example of a circular flow of $\mathcal{B} = \{\beta_1, \beta_2\}$. The chemical transformation from $\beta_1$ to $\beta_2$ is in $c^2$.}
    \label{fig:ThermoDefinition}
\end{figure}  

A well-established formalism to represent a network of systems is graph theory \cite{bondy1976graph}: examples of network design theories based on graphs are electrical networks \cite{freitas2020stochastic}, hydraulic networks \cite{deuerlein2008decomposition}, and multiagent systems \cite{mesbahi2010graph}. Since the system in Fig. \ref{fig:ThermoDefinition} can be seen as a network of thermodynamic compartments connected through the material flow, we will use graph theory to formulate the system in Fig. \ref{fig:ThermoDefinition} as a network and then state the definition of a thermodynamical material network.   

\begin{definition}[\hspace{1sp}\cite{bondy1976graph}]
A directed graph $D$ or \emph{digraph} is a graph identified by a set of $n_v$ \emph{vertices} $\{v_1, v_2, \dots, v_{n_v}\}$ and a set of $n_a$ \emph{arcs} $\{a_1, a_2, \dots, a_{n_a}\}$ that connect the vertices. A digraph $D$ in which each vertex or arc is associated with a \emph{weight} is a \emph{weighted digraph}. 
\end{definition}

Let $c^k_{i,j}$ be the $k$-th thermodynamic compartment through which the material flows from compartment $i$ to compartment $j$.
\begin{definition}[Thermodynamical material network]\label{def:TMN}
 A \emph{thermodynamical material network} (TMN) is a set $\mathcal{N}$ of connected thermodynamic compartments, that is, 
\begin{equation}\label{def:TMNset}
\begin{gathered}
\mathcal{N} = \left\{c^1_{1,1}, \dots, c^{k_v}_{k_v,k_v}, \dots, c^{n_v}_{n_v,n_v}, \right. \\ 
\left. c^{n_v+1}_{i_{n_v+1},j_{n_v+1}}, \dots, c^{n_v+k_a}_{i_{n_v+k_a},j_{n_v+k_a}}, \dots, c^{n_c}_{i_{n_c},j_{n_c}}\right\}, 
\end{gathered}
\end{equation}
which transport, store, and transform a set of $n_{\beta}$ materials $\mathcal{B} = \{\beta_1, \dots, \beta_q, \dots, \beta_{n_\beta}\}$ and whose modeling is based on compartmental dynamical thermodynamics \cite{haddad2019dynamical}. 
\end{definition}
Specifically, $\mathcal{N} = \mathcal{R} \cup \mathcal{T}$, where $\mathcal{R} \subseteq \mathcal{N}$ is the subset of compartments that \emph{store}, \emph{transform}, or \emph{use} the target material, and $\mathcal{T} \subseteq \mathcal{N}$ is the subset of compartments that \emph{move} the target material between the compartments belonging to $\mathcal{R} \subseteq \mathcal{N}$. A net $\mathcal{N}$ is associated with its weighted \emph{mass-flow digraph} $M(\mathcal{N})$, which is a weighted digraph whose vertices are the compartments $c^k_{i,j} \in \mathcal{R}$ and whose arcs are the compartments $c^k_{i,j} \in \mathcal{T}$. A vertex also results from the intersection of 3 or more arcs. For vertex-compartments $c^k_{i,j} \in \mathcal{R}$ it holds that $i = j$, whereas for arc-compartments $c^k_{i,j} \in \mathcal{T}$ it holds that $i \neq j$. The weight assigned to a vertex-compartment $c^k_{i,j} \in \mathcal{R}$ is identified by the mass stock $m_k$ within the corresponding compartment, whereas the weight assigned to an arc-compartment $c^k_{i,j} \in \mathcal{T}$ is the mass flow rate $\dot{m}_{i,j}$ from the vertex-compartment $c^i_{i,i} \in \mathcal{R}$ to the vertex-compartment $c^j_{j,j} \in \mathcal{R}$. Hence, the orientation of an arc-compartment $c^k_{i,j} \in \mathcal{T}$ is given by the direction of the material flow. The superscripts $k_v$ and $k_a$ in (\ref{def:TMNset}) are the $k$-th vertex and the $k$-th arc, respectively, while $n_c$ and $n_v$ are the total number of compartments and vertices, respectively. Since $n_a$ is the total number of arcs, it holds that $n_c = n_v + n_a$. 

\begin{definition}[Compartmental diagram]
The \emph{compartmental diagram} of the network (\ref{def:TMNset}) depicts the thermodynamic compartments $c^k_{i,j}$ and the arrows of material flows along with the material class $\beta_q \in \mathcal{B}$.
\end{definition}
\begin{definition}[Compartmental digraph]
The \emph{compartmental digraph} of the network (\ref{def:TMNset}) is a weighted digraph with arcs and vertices labeled with the corresponding compartmental nomenclature $c^k_{i,j}$.  
\end{definition}
Figure \ref{fig:TMN-Example1} shows an example of $\mathcal{N} = \left\{ c^1_{1,1}, c^2_{2,2}, c^3_{1,2}\right\}$ with $\mathcal{B} = \{\beta_1, \beta_2\}$, $n_c = 3$, $n_v = 2$, and $n_a = 1$ depicted using a compartmental diagram (top), a compartmental digraph (middle), and a mass-flow digraph $M(\mathcal{N})$ (bottom).     
\begin{figure}
    \centering
    \includegraphics[width=0.5\textwidth]{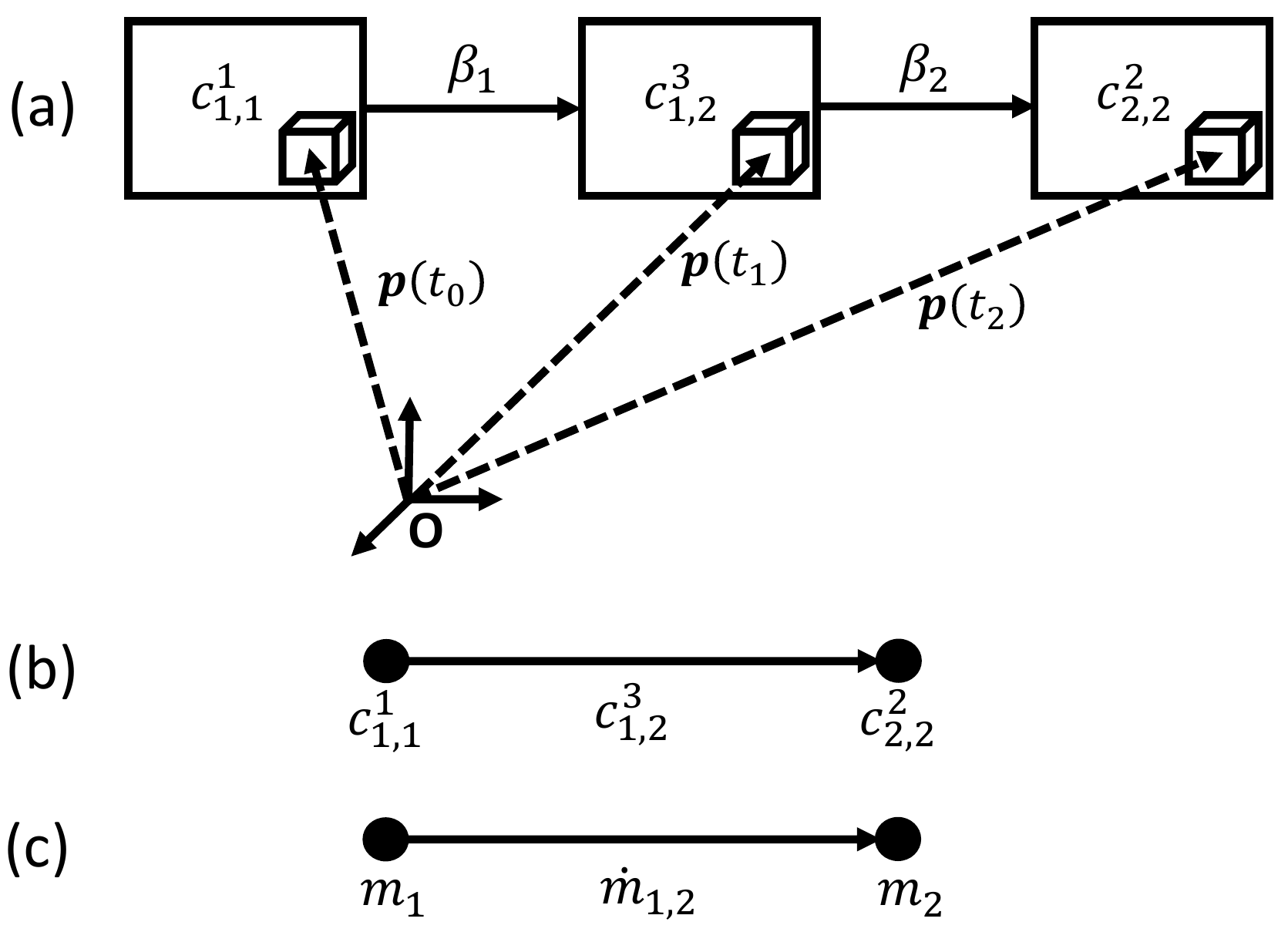}
    \caption{Graphical representations of $\mathcal{N} = \left\{ c^1_{1,1}, c^2_{2,2}, c^3_{1,2}\right\}$: (a) compartmental diagram, (b) compartmental digraph, and (c) mass-flow digraph.}
    \label{fig:TMN-Example1}
\end{figure}

Next, we introduce a few more definitions from graph theory. The reason will be clarified afterwards.
\notsure{
\begin{definition}[\hspace{1sp}\cite{bondy1976graph}]\label{def:directWalk}
A \emph{directed walk} in $D$ is a finite non-null sequence $W = (v_0, a_1, v_1, a_2, \dots, a_l, v_l)$ whose terms are alternatively vertices and arcs such that, for $i = 1, 2, \dots, l$, the arc $a_i$ has \emph{head} $v_i$ and \emph{tail} $v_{i-1}$. The integer $l$ is the \emph{length} of $W$, while the vertices $v_0$ and $v_l$ are the \emph{origin} and the \emph{terminus} of $W$, respectively.   
\end{definition}
\begin{definition}[\hspace{1sp}\cite{bondy1976graph}]
If the sequence of arcs $a_1, a_2, \dots, a_l$ of a directed walk $W$ are distinct, then $W$ is a \emph{directed trial}.
\end{definition}
\begin{definition}[\hspace{1sp}\cite{bondy1976graph}]
A directed trial is \emph{closed} if it has positive length and its origin and terminus are the same, i.e., $v_0 = v_l$.
\end{definition}
\begin{definition}[\hspace{1sp}\cite{bondy1976graph}]
A closed directed trial whose origin and internal vertices are distinct is a \emph{directed cycle} $\phi$.
\end{definition}
}
Summarizing, a directed cycle $\phi$ is a directed walk $W$ (Definition \ref{def:directWalk}) in which the arcs are distinct, the origin and the internal vertices are distinct, the origin and the terminus are the same, and $l > 0$. 

The reason for introducing these definitions is that it is now apparent that the requirement for material circularity (\ref{eq:circularityCond}) translates into requiring that the mass-flow digraph $M(\mathcal{N})$ must be a directed cycle $\phi$. 
\begin{remark}
Consider the network (\ref{def:TMNset}) with a mass-flow digraph $M(\mathcal{N})$. Then, the flow of the set of materials $\mathcal{B} = \{\beta_1, \dots, \beta_q, \dots, \beta_{n_\beta}\}$ is \emph{thermodynamically circular} if $M(\mathcal{N})$ is a directed cycle $\phi$.
\end{remark}
The \emph{mass-flow matrix} $\bm{\Gamma}(\mathcal{N})$ associated with the network (\ref{def:TMNset}) is given by
\begin{equation}\label{eq:gammaDef}
\begin{split}
\bm{\Gamma}(\mathcal{N}) & =
\begin{bmatrix}
\gamma_{1,1} & \dots & \gamma_{1,n_v} \\
    \vdots & \ddots & \vdots \\
\gamma_{n_v,1} &  \dots & \gamma_{n_v,n_v} 
\end{bmatrix}
\\ 
& = 
\begin{bmatrix}
m_1 & \dot{m}_{1,2} & \dots & \dot{m}_{1,n_v} \\
\dot{m}_{2,1} & m_2 & \dots & \dot{m}_{2,n_v} \\
\vdots & \ddots & \ddots & \vdots \\
\dot{m}_{n_v,1} & \dot{m}_{n_v,2} & \dots & m_{n_v} 
\end{bmatrix},
\end{split} 
\end{equation}
whose entries along the diagonal are the weights of the vertex-compartments $c^k_{i,j} \in \mathcal{R}$ (i.e., mass stocks) and whose off diagonal entries are the weights of the arc-compartments $c^k_{i,j} \in \mathcal{T}$ (i.e., mass flow rates). Hence, $\bm{\Gamma}(\mathcal{N})$ is a square matrix of size $n_v \times n_v$ with nonnegative real entries, i.e., $\bm{\Gamma} \in \overline{\mathbb{R}}^{n_v \times n_v}_{+}$. 

The mass conservation principle \cite{kaminski2017introduction} establishes the relationship between the entries of $\bm{\Gamma}(\mathcal{N})$, namely,
\begin{equation}\label{eq:compMassBalance}
\frac{\text{d}}{\text{d}t}m_{k} = \sum_{i =1}^{n_v}\dot{m}_{i,k} - \sum_{j = 1}^{n_v} \dot{m}_{k,j},
\end{equation}
which can be further written in terms of the entries of $\bm{\Gamma}(\mathcal{N})$ as
\begin{equation}
\frac{\text{d}}{\text{d}t}\gamma_{k,k} = \sum_{i =1, i \neq k}^{n_v}\gamma_{i,k} - \sum_{j =1, j \neq k}^{n_v} \gamma_{k,j},
\end{equation}
or, equivalently, in vector form as
\begin{equation}\label{eq:massBalanceVectorialForm}
\begin{split}
\frac{\text{d}}{\text{d}t}\bm{m} & = \frac{\text{d}}{\text{d}t} 
\begin{bmatrix}
m_1 \\
m_2 \\
\vdots \\
m_{n_v}
\end{bmatrix}
= \frac{\text{d}}{\text{d}t}
\begin{bmatrix}
\gamma_{1,1} \\
\gamma_{2,2} \\
\vdots \\
\gamma_{n_v,n_v} 
\end{bmatrix}  
\\ 
& = 
\begin{bmatrix}
\sum\limits_{\substack{i =1 \\ i \neq 1}}^{n_v}\gamma_{i,1} - \sum\limits_{\substack{j =1 \\ j \neq 1}}^{n_v} \gamma_{1,j} \\
\sum\limits_{\substack{i =1 \\ i \neq 2}}^{n_v}\gamma_{i,2} - \sum\limits_{\substack{j =1 \\ j \neq 2}}^{n_v} \gamma_{2,j} \\
\vdots \\
\sum\limits_{\substack{i =1 \\ i \neq n_v}}^{n_v}\gamma_{i,n_v} - \sum\limits_{\substack{j =1 \\ j \neq n_v}}^{n_v} \gamma_{n_v,j}
\end{bmatrix}
.
\end{split}
\end{equation}
Therefore,
\begin{equation}\label{eq:constraintOfStocks}
\begin{split}
\bm{m} & =
\begin{bmatrix}
\gamma_{1,1} \\
\gamma_{2,2} \\
\vdots \\
\gamma_{n_v,n_v} 
\end{bmatrix} 
 \\ 
 & = 
\begin{bmatrix}
\int_{t_0}^t \left(\sum\limits_{\substack{i =1 \\ i \neq 1}}^{n_v}\gamma_{i,1} - \sum\limits_{\substack{j =1 \\ j \neq 1}}^{n_v} \gamma_{1,j}\right) \text{d}\tau\\
\int_{t_0}^t \left(\sum\limits_{\substack{i =1 \\ i \neq 2}}^{n_v}\gamma_{i,2} - \sum\limits_{\substack{j =1 \\ j \neq 2}}^{n_v} \gamma_{2,j}\right) \text{d}\tau\\
\vdots \\
\int_{t_0}^t \left(\sum\limits_{\substack{i =1 \\ i \neq n_v}}^{n_v}\gamma_{i,n_v} - \sum\limits_{\substack{j =1 \\ j \neq n_v}}^{n_v} \gamma_{n_v,j}\right) \text{d}\tau
\end{bmatrix}
, \quad t \geq t_0.
\end{split}
\end{equation}
\begin{remark}
In line with the standard nomenclature adopted in thermal engineering, the mass accumulation or depletion in vertex-compartments is denoted as $\frac{\text{d}}{\text{d}t}m$ and not as $\dot{m}$ \cite{kaminski2017introduction}. Indeed, the latter indicates a mass flow and involves a mass transfer between two vertex-compartments. The first form will be referred to as \emph{mass accumulation-depletion}, whereas the second form will be referred to as \emph{mass flow rate}. The SI unit is kg/s for both quantities. 
\end{remark}

Leveraging the mass-flow matrix (\ref{eq:gammaDef}), we now define the circularity indicator $\lambda$.
\begin{definition}[Graph-based circularity indicator]\label{def:CircIndicator}
The \emph{graph-based circularity indicator} $\lambda(\bm{\Gamma}) \in [0, 1]$ of the network (\ref{def:TMNset}) associated with the mass-flow matrix (\ref{eq:gammaDef}) is 
\begin{equation}\label{eq:lambda}
\lambda(\bm{\Gamma}) = \frac{\sum\limits^{n_\phi}_{k = 1} \text{CM}(\phi_k)}{\sum\limits^{n_\phi}_{k = 1} \text{CM}(\phi_k) + \sum\limits_{\gamma_{i,j} \in \mathcal{Q}}\gamma_{i,j}},
\end{equation}
where $n_\phi$ is the number of directed cycles in $M$,
\begin{equation}
\mathcal{Q} = \{\gamma_{i,j}|\dot{m}_{i,j} \text{ does not belong to any directed cycle}\},
\end{equation}
and
\begin{equation}
\text{CM}(\phi) = \frac{1}{l} \sum\limits_{\gamma_{i,j} \in \mathcal{Y}} \gamma_{i,j} 
\end{equation}
is the \emph{cycle mean} of $\phi$, with 
\begin{equation}
\mathcal{Y} = \{\gamma_{i,j}|\dot{m}_{i,j} \in \phi \}.
\end{equation}
\end{definition}
The next section proposes a methodology for the design of circular material flows using the definitions given above.

\section{Design Methodology}\label{sec:Method}
The proposed methodology is outlined in Fig. \ref{fig:Methodology} and involves three main steps; its output is a TMN, its goal is designing the flow of the material set $\mathcal{B} = \{\beta_1, \dots, \beta_q, \dots, \beta_{n_\beta}\}$, and it has two optional extensions at Steps 2 and 3 indicated by the black arrows.     
\begin{figure}
    \centering
    \includegraphics[width=0.49\textwidth]{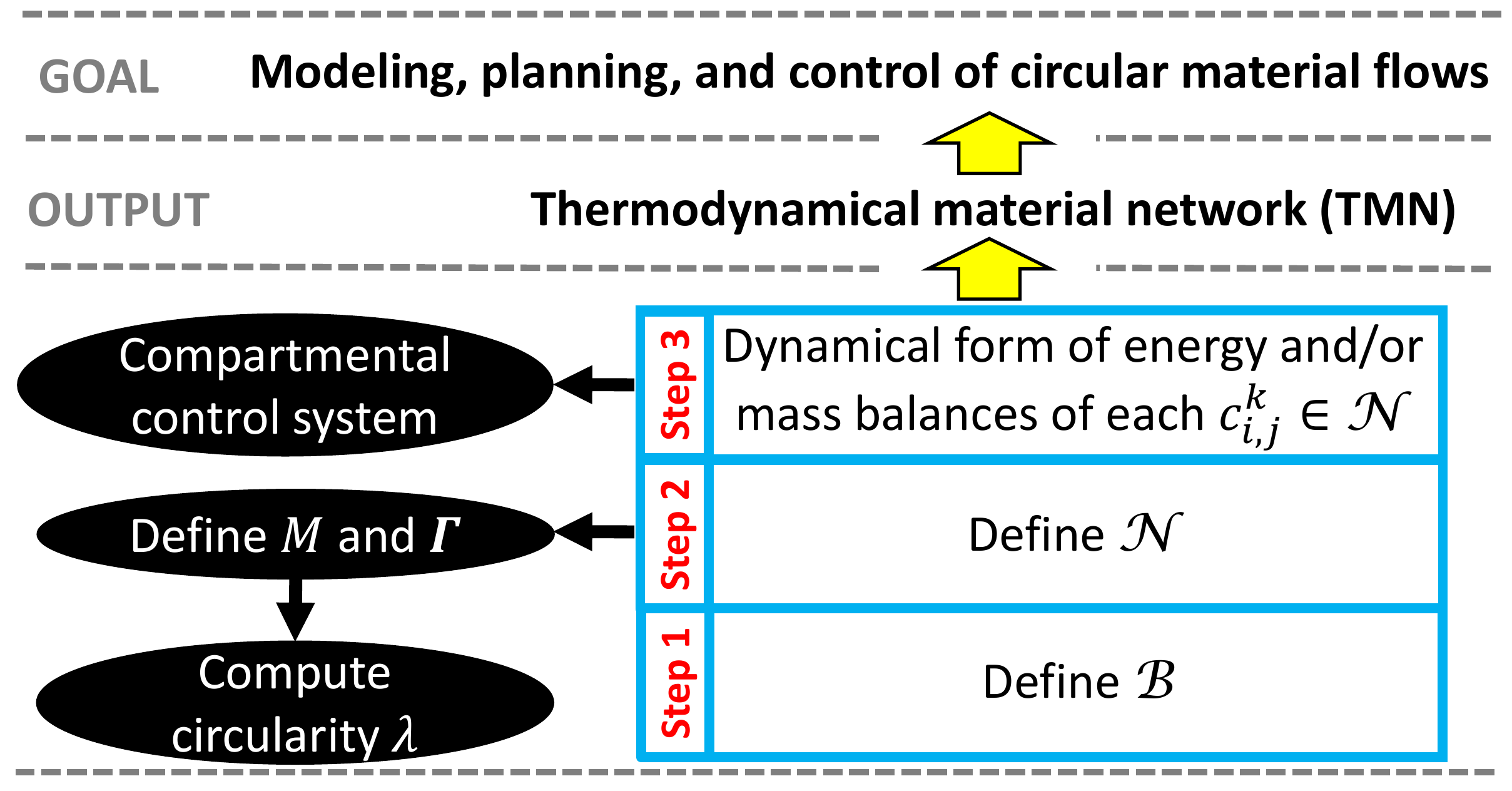}
    \caption{Proposed methodology to design a circular flow of $\mathcal{B}$.}
    \label{fig:Methodology}
\end{figure}
Specifically, the first step is the choice of the material set $\mathcal{B}$ to be circulated. Then, the network (\ref{def:TMNset}) is defined and depicted as needed using a compartmental diagram, a compartmental digraph and/or a mass-flow digraph (an example is in Fig. \ref{fig:TMN-Example1}). Moreover, its circularity can be measured by computing the indicator (\ref{eq:lambda}) as indicated with the black arrow on the left-hand side, which requires to preliminarily define the mass-flow matrix (\ref{eq:gammaDef}). The third and last step consists of applying to each compartment $c^k_{i,j} \in \mathcal{N}$ the dynamical form of the compartmental energy balance, that is, 
\begin{equation}\label{eq:CompPowerBal}
\frac{\text{d}E_k}{\text{d}t} = \frac{\text{d}}{\text{d}t}(K_k + U_k + P_k) = \dot{Q}_k - \dot{W}_k, 
\end{equation} 
and/or the compartmental mass balance (\ref{eq:compMassBalance}). Equation (\ref{eq:CompPowerBal}) is a power balance in which $E_k$, $K_k$, $U_k$, and $P_k$ are the total energy, the kinetic energy, the internal energy, and the potential energy of the compartment $k$, respectively, and $\dot{Q}_k$ and $\dot{W}_k$ are the heat flow and the work flow exchanged by the compartment $k$ with the surroundings. 

Once the dynamics of the compartments are defined, it is possible to implement a compartmental control system as indicated by the black arrow in Step 3 if the power balance (\ref{eq:CompPowerBal}) and the mass balance (\ref{eq:compMassBalance}) are written in state-space form
\begin{equation}
\dot{\bm{x}}_k = \bm{F}(\bm{x}_k, \bm{u}_k), 
\end{equation}     
where $\bm{x}_k \in \mathbb{R}^{n_{k}}$ is the state vector of the $k$-th compartment, $\bm{u}_k \in \mathbb{R}^{z_{k}}$ is the control input of the $k$-th compartment, and $\bm{F} : \mathbb{R}^{n_{k}} \times \mathbb{R}^{z_{k}} \rightarrow \mathbb{R}^{n_{k}}$ is a continuous nonlinear function.

\section{Illustrative Examples}\label{sec:Examples}
\subsection{Circularity Calculation}\label{sub:FirstExample}
This example demonstrates the calculation of the circularity indicator (\ref{eq:lambda}) considering the net $\mathcal{N}$ depicted in Fig. \ref{fig:Example1-circularity}, where $\alpha \in [0, 1]$. 
\begin{figure}
    \centering
    \includegraphics[width=0.49\textwidth]{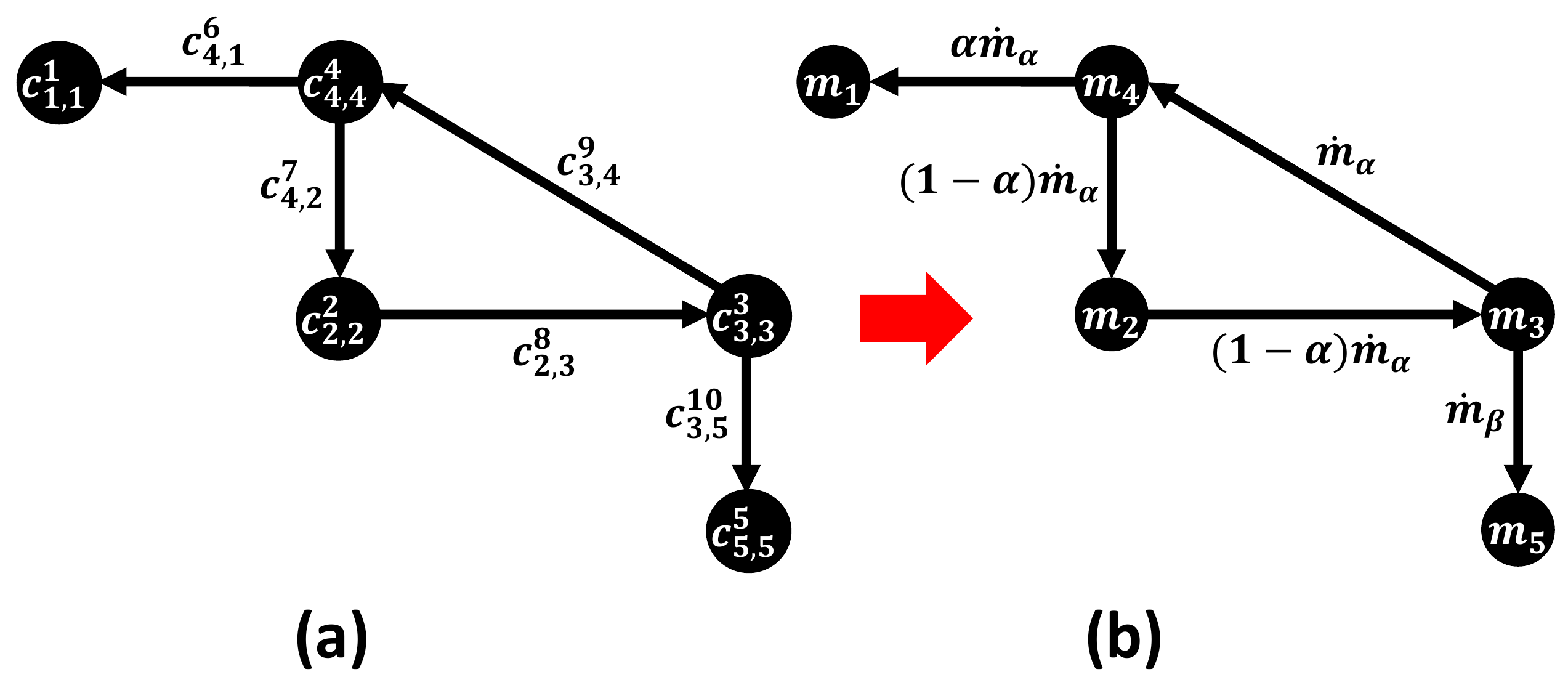}
    \caption{Compartmental digraph in (a) and corresponding mass-flow digraph in (b) considered in Example \ref{sub:FirstExample} with $\alpha \in [0, 1]$.}
    \label{fig:Example1-circularity}
\end{figure}
Assume $m_i = 0|_{i = 2, 3, 4}$. Hence, the mass balance (\ref{eq:constraintOfStocks}) is respected in $c^2_{2,2}$ and $c^4_{4,4}$, while in $c^3_{3,3}$ it yields 
\begin{equation}
(1-\alpha)\dot{m}_{\alpha} = \dot{m}_\beta + \dot{m}_{\alpha} \Rightarrow \dot{m}_\beta = - \alpha \dot{m}_\alpha < 0,    
\end{equation}
which is nonphysical. Therefore, the direction of $c^{10}_{3,5}$ must be inverted as $c^{10}_{5,3}$ to get $\dot{m}_\beta = \alpha \dot{m}_\alpha$. The remaining stocks $m_1$ and $m_5$ follow from the mass balance as $m_1(t) = \int_0^t \alpha \dot{m}_\alpha \text{d}\tau$ and $m_5(t) = - \int_0^t \alpha \dot{m}_\alpha \text{d}\tau$, with the latter requiring that $\alpha \dot{m}_\alpha t \leq m_{5,0}$, where $m_{5,0}$ is the initial stock in $c^5_{5,5}$. With this, all the stocks and flows different from zero are functions of $\dot{m}_\alpha$. 

To calculate the circularity $\lambda$ given by (\ref{eq:lambda}), we observe that for $\alpha = 1$ the network has no cycles, and hence $\lambda = 0$. In contrast, if $\alpha = 0$, then $\mathcal{Q} = \emptyset$, and hence $\lambda = 1$. Finally, for $0 < \alpha < 1$, the net has a cycle with
\begin{equation}
\text{CM}(\phi) = \frac{\dot{m}_\alpha(3 - 2\alpha)}{3},    
\end{equation}
and hence, 
\begin{equation}
\lambda(\alpha) = \frac{3 - 2\alpha}{3+4\alpha}. 
\end{equation}
In summary,
\begin{equation}
\lambda(\alpha) =
\begin{cases}
      0, & \quad \alpha = 1 \\
      \frac{3 - 2\alpha}{3+4\alpha}, & \quad 0 \leq \alpha < 1.  
\end{cases}
\end{equation}  
Figure \ref{fig:LambdaOfAlpha} shows $\lambda(\alpha)$ as a function of $\alpha$.
\begin{figure}
    \centering
    \includegraphics[width=0.49\textwidth]{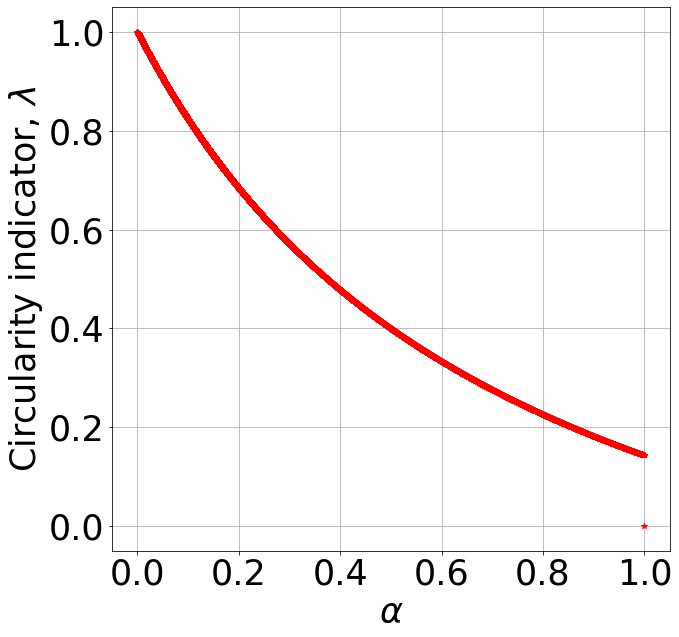}
    \caption{Circularity indicator (\ref{eq:lambda}) as a function of $\alpha$ for the network in Fig. \ref{fig:Example1-circularity}.}
    \label{fig:LambdaOfAlpha}
\end{figure}

\subsection{Subsystem of Bio-Methane Supply Chain}\label{sub:SecondExample}
This example demonstrates the application of the design methodology on a subsystem of a biomethane supply chain. Specifically, the subsystem involves three stages of the biomass life-cycle (see Fig. \ref{fig:BioChain}): the biomass hub, the truck to transport the biomass, and the anaerobic digestion plant to covert the biomass into biogas. Moreover, the anaerobic digestion plant is divided into two sub-compartments: $P_r$ is the plant reservoir of biomass and $P_d$ is the plant digester. We assume that $\bm{G}(t) = [x_G(t), 0, 0]^\top$ (i.e., the material motion is along $x$ only), the position of the hub is $x_h$, $\bm{O} \equiv x_h$, $x_p$ is the position of the chemical plant, and that the sizes of the hub and the plant are negligible compared to $H = x_p - x_h$.      
\begin{figure}
    \centering
    \includegraphics[width=0.49\textwidth]{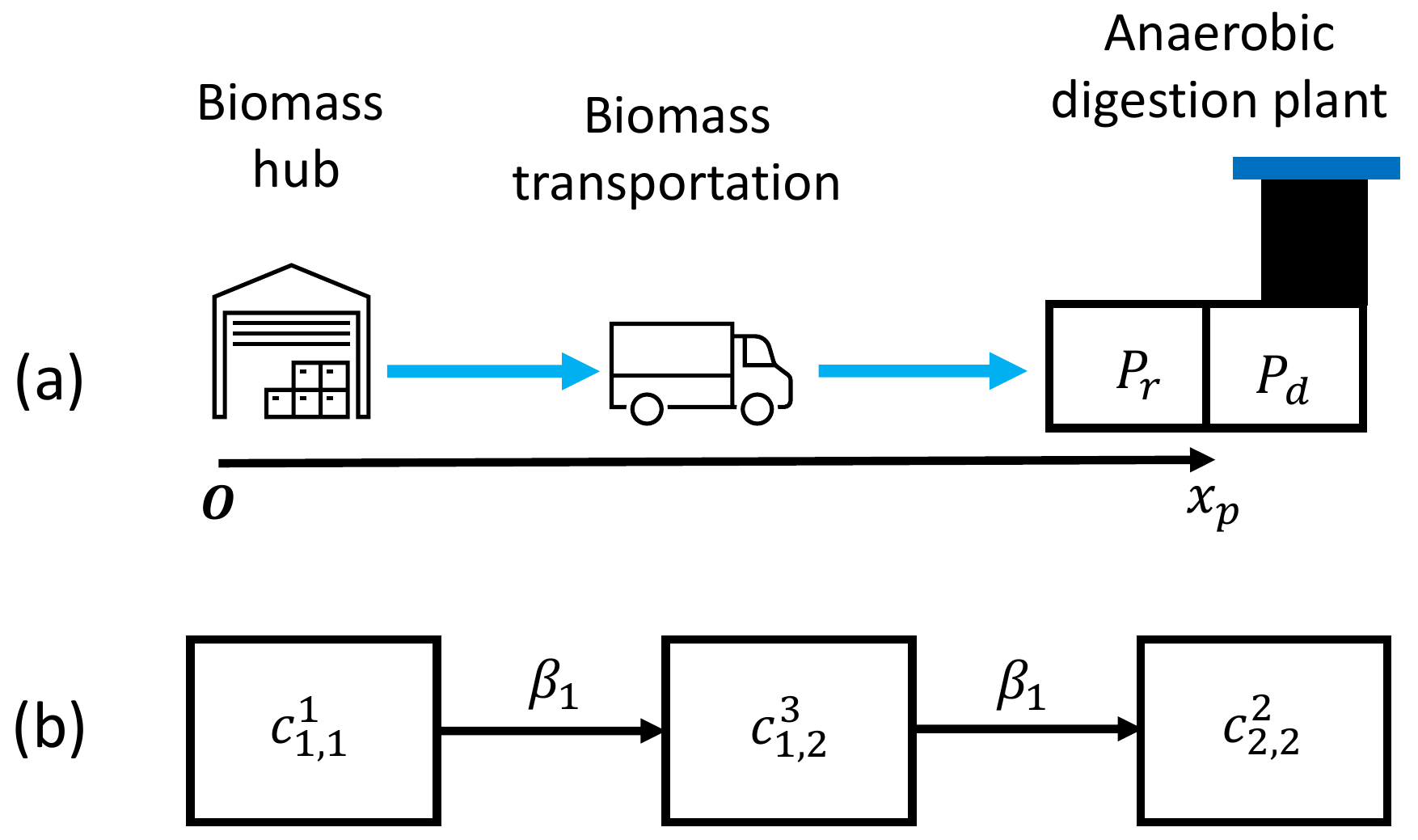}
    \caption{The subsystem under study in Example \ref{sub:SecondExample}: (a) physical representation and (b) compartmental diagram.}
    \label{fig:BioChain}
\end{figure}

The first step of the methodology requires to choose the set of materials of interest, i.e., $\mathcal{B}$; in this case, $\mathcal{B} = \{\beta_1\}$, where $\beta_1$ is the biomass (the details of its chemical composition are not considered in this example). The second step requires to define the net (\ref{def:TMNset}); for this example, $\mathcal{N} = \{c^1_{1,1}, c^2_{2,2}, c^3_{1,2}\}$, $n_c = 3$, $n_v = 2$, and $n_a = 1$. The thermodynamic modeling of each compartment as required at Step 3 is as follows.

\subsubsection{$c^1_{1,1}$ (Biomass hub)} The process of exiting the solid biomass from the hub can be modeled more naturally as a discrete-time system rather than a continuous-time system since the material output flow is carried on in batches instead of as a continuous flow (as it would be with fluids). The discrete-time mass balance for the hub yields
\begin{equation}
m_1(n + 1) - m_1(n) = - \delta_{n_l}(n) m_l, \quad n \in \overline{\mathbb{Z}}_+,
\end{equation}     
where $\overline{\mathbb{Z}}_+$ is the set of nonnegative integers, $m_1$ is the mass stock inside the hub, $m_l$ is the truckload, $n_l$ is the loading time, and $\delta_{n_l}(n)$ is the Kronecker delta, that is,    
\begin{equation}
\delta_{n_l}(n) =
\begin{cases}
      0, & \quad n \neq n_l \\
      1, & \quad n = n_l. 
    \end{cases}
\end{equation}

\subsubsection{$c^3_{1,2}$ (Truck)}
For $n_l \leq n \leq n_u$, the biomass with mass $m_l$ is on the truck, it is transported to the plant, and it enters the plant at the unloading time $n_u$. In this example, we model the truck as the three-wheel vehicle shown in Fig. \ref{fig:GabicciniImage} \cite{GabicciniEsame}. 
\begin{figure}
    \centering
    \includegraphics[width=0.5\textwidth]{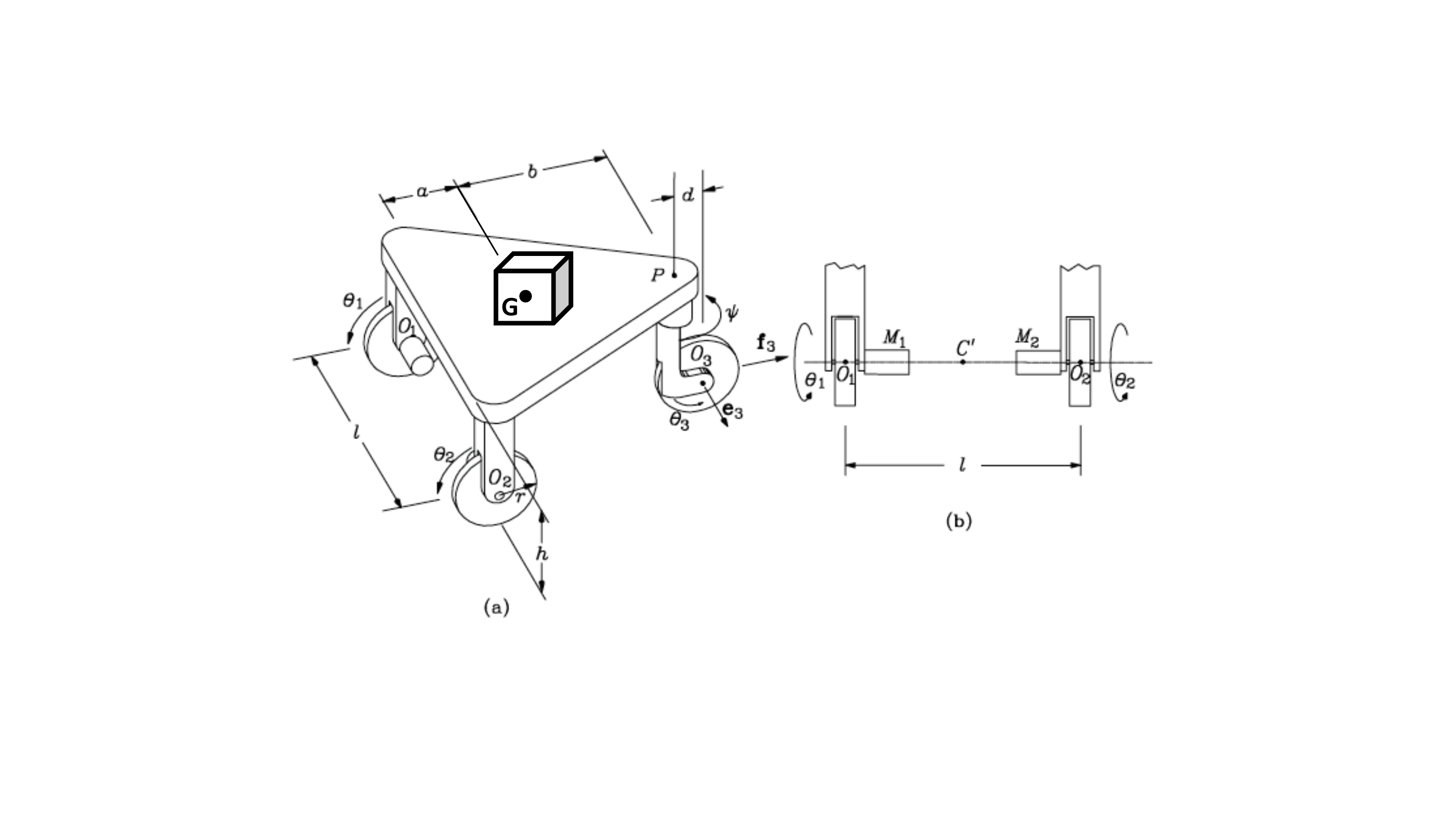}
    \caption{A three-wheel vehicle model of the truck, that is, compartment $c^3_{1,2}$ (modified from \cite{GabicciniEsame}).}
    \label{fig:GabicciniImage}
\end{figure}
The equations of motion of the vehicle can be derived using Lagrange's equations of motion (\hspace{1sp}\cite{siciliano2010robotics})
\begin{equation}\label{eq:LagrMech}
\frac{\text{d}}{\text{d}t}\left(\frac{\partial L}{\partial \dot{\bm{q}}}\right)^\top = \bm{\xi} + \left(\frac{\partial L}{\partial \bm{q}}\right)^\top,
\end{equation} 
where 
\begin{equation}\label{eq:LagranFunct}
L = K - P
\end{equation}
is the Lagrangian function of the mechanical system, $\bm{q} \in \mathbb{R}^d$ is the vector of generalized coordinates, $\dot{\bm{q}} \in \mathbb{R}^d$ is the vector of generalized velocities, and $\bm{\xi} \in \mathbb{R}^d$ is the vector of generalized forces associated with the generalized coordinates $\bm{q}$. Before we use the Lagrangian formulation (\ref{eq:LagrMech}) to define the dynamical equations of the three-wheel vehicle, we demonstrate that the Lagrangian formulation (\ref{eq:LagrMech}) can be derived from the dynamical form of the energy balance (\ref{eq:CompPowerBal}). This is a key result for this paper as it shows that Lagrangian mechanics respects Step 3 of the proposed methodology, which requires that we model each compartment using the dynamical form of the energy and/or the mass balances.     
\begin{proposition}\label{pro:LagMechFromEnBal}
Lagrange's equations of motion given by (\ref{eq:LagrMech}) can be derived from the dynamical form of the energy balance (\ref{eq:CompPowerBal}).  
\end{proposition}
\begin{proof}
For simplicity of exposition, consider the dynamical form of the energy balance (\ref{eq:CompPowerBal}) without the subscript $k$ specifying the $k$-th compartment, that is, 
\begin{equation}\label{eq:PowerBal}
\frac{\text{d}}{\text{d}t}(K + U + P) = \dot{Q} - \dot{W}.
\end{equation}
Since the heat flow $\dot{Q}$ and the internal energy $U$ are neglected in solid mechanics, (\ref{eq:PowerBal}) reduces to
\begin{equation}\label{eq:PowerBalMechanics} 
\frac{\text{d}}{\text{d}t}(K + P) = - \dot{W}.
\end{equation}
Note that in solid mechanics the potential energy $P$ corresponds to the conservative work done by the gravitational force, and hence (\hspace{1sp}\cite{siciliano2010robotics}), 
\begin{equation}\label{eq:derivePqdot}
\frac{\partial P}{\partial \dot{\bm{q}}} = 0.
\end{equation}
Taking the partial derivative on both sides of (\ref{eq:PowerBalMechanics}) with respect to $\dot{\bm{q}}$ and transposing the resulting equation yields
\begin{equation}\label{eq:powBalDeriveqdot}
\frac{\text{d}}{\text{d}t}\left[\frac{\partial(K + P)}{\partial \dot{\bm{q}}}\right]^\top = - \left(\frac{\partial \dot{W}}{\partial \dot{\bm{q}}}\right)^\top.
\end{equation} 
It now follows from (\ref{eq:derivePqdot}) that the term on the left-hand side of (\ref{eq:powBalDeriveqdot}) can be written in terms of the Lagrangian function (\ref{eq:LagranFunct}), and hence, (\ref{eq:powBalDeriveqdot}) can be rewritten as 
\begin{equation}\label{eq:rewriteWithL}
\frac{\text{d}}{\text{d}t}\left(\frac{\partial L}{\partial \dot{\bm{q}}}\right)^\top = - \left(\frac{\partial \dot{W}}{\partial \dot{\bm{q}}}\right)^\top. 
\end{equation}

Now, note that the term on the left-hand side of (\ref{eq:rewriteWithL}) is equal to the term on the left-hand side of the Lagrangian formulation (\ref{eq:LagrMech}), and hence,
\begin{equation}\label{eq:equalityRightTerms}
- \left(\frac{\partial \dot{W}}{\partial \dot{\bm{q}}}\right)^\top = \bm{\xi} + \left(\frac{\partial L}{\partial \bm{q}}\right)^\top. 
\end{equation} 
Therefore, (\ref{eq:rewriteWithL}) can be written as (\ref{eq:LagrMech}).
\end{proof}

For the three-wheel vehicle (\ref{eq:LagrMech}) (which can be derived from the power balance (\ref{eq:PowerBal}) as shown in Proposition \ref{pro:LagMechFromEnBal} and thus it respects Step 3 of our methodology) yields \cite{GabicciniEsame}
\begin{equation}\label{eq:TruckDynamics}
\bm{B}
\begin{bmatrix}
\ddot{\theta}_1 \\
\ddot{\theta}_2 \\
\end{bmatrix}
=
\begin{bmatrix}
\tau_1 \\
\tau_2 \\
\end{bmatrix}
,
\end{equation}
where $\tau_1$ and $\tau_2$ are the control torques of the engines $M_1$ and $M_2$, respectively, and 
\begin{equation}
\bm{B} = m_{vl}
\begin{bmatrix}
\left(\frac{a^2r^2}{l^2} + \frac{r^2}{4}\right) + \varepsilon  & \left(-\frac{a^2r^2}{l^2} + \frac{r^2}{4}\right) - \varepsilon \\
\left(-\frac{a^2r^2}{l^2} + \frac{r^2}{4}\right) - \varepsilon & \left(\frac{a^2r^2}{l^2} + \frac{r^2}{4}\right) + \varepsilon \\ 
\end{bmatrix}
, 
\end{equation}
where $m_{vl} = m_v + m_l$, $m_v$ is the vehicle mass, $\varepsilon = I_z \frac{r^2}{l^2}$, and $I_z$ is the principal moment of inertia of the vehicle with respect to its $z$-axis (note the dependence of the truck dynamics (\ref{eq:TruckDynamics}) on the truckload $m_l$). Once $\dot{\theta}_1$ and $\dot{\theta}_2$ are determined by integration of (\ref{eq:TruckDynamics}), $\dot{\psi}$ and $\dot{\theta}_3$ are given by
\begin{equation}\label{eq:nonActuatedJoints}
\begin{bmatrix}
\dot{\theta}_3 \\
\dot{\psi} 
\end{bmatrix}
=
\bm{F}
\begin{bmatrix}
\dot{\theta}_1 \\
\dot{\theta}_2 \\
\end{bmatrix}
,
\end{equation}
where 
\begin{equation}
\bm{F} =
\begin{bmatrix}
\frac{\text{cos }\psi}{2} - \alpha\text{sin }\psi & \frac{\text{cos }\psi}{2} + \alpha\text{sin }\psi \\
\rho\left(- \delta_{-} - \alpha\text{cos }\psi \right) & \rho\left(- \delta_{+} + \alpha\text{cos }\psi \right) \\
\end{bmatrix}
,
\end{equation}
with $\delta_{\pm} = \frac{\text{sin }\psi}{2} \pm \frac{d}{l}$, $\alpha = \frac{a+b}{l}$, and $\rho = \frac{r}{d}$. As mentioned above, here we assume that the material motion is along the $x$-axis only, and hence, $\psi(t) = 0$, $\theta_1(t) = \theta_2(t) = \theta(t)$, and $\tau_1(t) = \tau_2(t) = \tau (t)$. Moreover, $x_{G}(t) = x_{O_2}(t) + a$, where $x_{O_2}(t) = r\theta (t)$, and hence, $\ddot{x}_{G}(t) = \ddot{x}_{O_2}(t)$ and $\ddot{x}_{G}(t) = r\ddot{\theta}(t)$. Thus, the two equations of the system (\ref{eq:TruckDynamics}) are identical and give the dynamics of $\bm{G}$
\begin{equation}
m_{vl} r \ddot{x}_G = 2 \tau,
\end{equation}
which is independent of $I_z$ as the motion is purely translational. Moreover, (\ref{eq:nonActuatedJoints}) gives $\dot{\theta}_3 (t) = \dot{\theta}(t)$ and $\dot{\psi}(t) = 0$.

\subsubsection{$c^2_{2,2}$ (Chemical plant)}    
As showed in Fig. \ref{fig:BioChain}, the anaerobic digestion plant is divided into $P_r$ and $P_d$. The mass balance of $P_r$ in discrete-time yields
\begin{equation}
m_r(n+1) - m_r(n) = \delta_{n_u}(n)m_l, \quad n < n_d, \quad n \in \overline{\mathbb{Z}}_+,
\end{equation}
where $n_d$ is the time in which the reaction in $P_d$ begins and $m_r$ is the mass inside $P_r$. 

For $n \geq n_d$, the sub-compartment $P_r$ becomes a continuous-time system which supplies a continuous flow to the digester $P_d$. The digester $P_d$ is modeled as a continuous stirred tank reactor. Specifically, the anaerobic digestion occurring inside $P_d$ is a four-state dynamical system, which results from the mass balance of the species involved and it assumes a two-stage reaction. Namely, first, the organic substrate $S_1(t)$ is degraded into volatile fatty acids $S_2(t)$ by acidogenic bacteria $X_1(t)$, and then the methanogenic bacteria $X_2(t)$ consume the volatile fatty acids to produce methane CH$_4$ and carbon dioxide CO$_2$ \cite{campos2019hybrid}. The set of four ordinary differential equations of the anaerobic digestion resulting from the mass balance are given by \cite{campos2019hybrid}
\begin{gather}
\begin{align}
\dot{X}_1(t) & = \left[\mu_1(S_1(t)) - \alpha D_1(t)\right]X_1(t), \notag\\ 
& \quad X_1(0) = X_{1,0}, \quad t \geq 0, \\
\dot{X}_2(t) & = \left[\mu_2(S_2(t)) - \alpha D_2(t)\right]X_2(t), \notag\\
&\quad X_2(0) = X_{2,0}, \\
\dot{S}_1(t) & = D_3(t) \left(S_{1\text{in}} - S_1(t)\right) - k_1 \mu_1(S_1(t))X_1(t), \notag\\
& \quad S_1(0) = S_{1,0}, \\
\dot{S}_2(t) & = D_4(t) \left(S_{2\text{in}} - S_2(t)\right) \notag\\
& + k_2 \mu_1(S_1(t))X_1(t) - k_3 \mu_2(S_2(t))X_2(t), \notag\\
& \quad S_2(0) = S_{2,0},
\end{align}
\end{gather}
where
\begin{equation}
\mu_1(S_1(t)) = \mu_{1\text{max}} \frac{S_1(t)}{S_1(t) + K_{\text{S1}}},
\end{equation}
$D_j(t)|_{j=1,2,3,4}$ is the dilution rate for the $j$-th state,
\begin{equation}
\mu_2(S_2(t)) = \mu_{2\text{max}} \frac{S_2(t)}{S_2(t) + K_{\text{S2}} + (S_2(t)/K_{\text{I2}})^2}, 
\end{equation}
and $\alpha, S_{1\text{in}}$, $S_{2\text{in}}$, $k_1$, $k_2$, $k_3$, $\mu_{1\text{max}}$, $\mu_{2\text{max}}$, $K_{\text{S1}}$, $K_{\text{S2}}$, and $K_{\text{I2}}$ are specific constants detailed in \cite{bernard2001dynamical,campos2019hybrid}. The input material flow to $P_d$ supplied by $P_r$ is 
\begin{equation}
\dot{\bm{m}}_{d,\text{in}}(\bm{u}(t)) = \rho_b \bm{u}(t) V_d, 
\end{equation}
where $\rho_b$ and $V_d$ are the biomass density and the digester working volume, respectively, while the biomethane flow $q_M$ produced by the anaerobic digestion is $q_M(S_2(t), X_2(t)) = k_6 \mu_2(S_2(t)) X_2(t)$ \cite{bernard2001dynamical}. Hence, for $n \geq n_d$, the mass balance of $P_r$ yields
\begin{equation}
\frac{\text{d}m_r(t)}{\text{d}t} = - \rho_b V_d \sum_1^4 D_j(t).  
\end{equation}

\subsubsection{Compartmental control}
Here we introduce two compartmental controllers following the black arrow in Step 3 of our methodology (Fig. \ref{fig:Methodology}): one for $c^3_{1,2}$ and one for $c^2_{2,2}$. Note that this can be done at this stage because the dynamics of the compartments are known.
  
The motion of the truck must satisfy $x_G(t_u) = H + a$ and $\dot{x}_G(t_u) = 0$, where $t_u$ is the unloading time. In this way, the truck reaches the plant at the unloading time $t_u$ with null speed. A simple open-loop control law that satisfies the requirements is
\begin{equation}
\tau(t) = u(t) =
\begin{cases}
      \overline{\tau}, & \quad t_l \leq t < \frac{t_u}{2} \\
      -\overline{\tau}, & \quad \frac{t_u}{2} \leq t < t_u,
    \end{cases}
\end{equation}  
where 
\begin{equation}
\overline{\tau} = \frac{2(m_v + m_l)rH}{t_u^2}.
\end{equation}

To regulate the digester to the desired working point, that is, the equilibrium ``SS6'' in \cite{campos2019hybrid}, a nonlinear control system was designed and implemented. Specifically, we implemented the control law of Haddad \emph{et al.} \cite{haddad2015finite} because, for dynamical systems of the form
\begin{equation}\label{eq:affineFormGeneral}
\dot{\bm{x}}(t) = \bm{f}(\bm{x}(t)) + \bm{G}(\bm{x}(t))\bm{u}(t), \quad \bm{x}(0) = \bm{x}_0, 
\end{equation}
where $\bm{f} : \mathbb{R}^{n_{k}} \rightarrow \mathbb{R}^{n_{k}}$ and $\bm{G} : \mathbb{R}^{n_{k}} \rightarrow \mathbb{R}^{n_{k} \times z_{k}}$ are continuous functions with $\bm{f}(0) = 0$, it guarantees the stabilization of the zero solution $\bm{x}(t) \equiv \bm{0}$ of (\ref{eq:affineFormGeneral}) in a finite-time, where $\bm{0} \in \mathbb{R}^{n_{k}}$ is a vector of zeros. In our case, $\bm{x}(t) = [X_1(t), S_1(t), X_2(t), S_2(t)]^\top$ and $\bm{u}(t) = [D_1(t), D_2(t), D_3(t), D_4(t)]^\top$. We designed the controller by choosing the functions $\bm{L}_2(\tilde{\bm{x}})$, $\bm{R}^{-1}_2(\tilde{\bm{x}})$, and $V(\tilde{\bm{x}})$ (in the notation of \cite{haddad2015finite}) as
\begin{equation}
\bm{L}_2(\tilde{\bm{x}}) = 2[\bm{f}^\top(\tilde{\bm{x}})\bm{G}(\tilde{\bm{x}})],
\end{equation}
\begin{equation}
\bm{R}^{-1}_2(\tilde{\bm{x}}) = \bm{G}^{-1}(\tilde{\bm{x}})[\bm{G}^{\top}(\tilde{\bm{x}})]^{-1},
\end{equation}
and 
\begin{equation}
V(\tilde{\bm{x}}) = p^{\frac{2}{3}}(\tilde{\bm{x}}^\top\tilde{\bm{x}})^{\frac{2}{3}},
\end{equation}         
where $\tilde{\bm{x}}(t)$ is the state vector translated to have the desired working point corresponding to the zero solution. This choice requires that $\bm{G}(\tilde{\bm{x}}(t))^{-1}$ exists, but it has the following benefits: first, the condition (33) of \cite{haddad2015finite} simplifies significantly so that it can be easily checked for systems with complex expressions of $\bm{f}(\tilde{\bm{x}}(t))$ and $\bm{G}(\tilde{\bm{x}}(t))$ since now it only depends on $V(\tilde{\bm{x}})$; second, since $\bm{f}(0) = 0$, the condition (34) of \cite{haddad2015finite} is satisfied regardless of the expressions of $\bm{f}(\tilde{\bm{x}}(t))$ and $\bm{G}(\tilde{\bm{x}}(t))$; and third, the dynamics of the closed-loop reduce to the gradient system
\begin{equation}
\dot{\tilde{\bm{x}}}(t) = -\frac{1}{2}[V^\prime(\tilde{\bm{x}}(t))]^\top, \quad \tilde{\bm{x}}(0) = \tilde{\bm{x}}_0,
\end{equation}
which is easy to interpret: the state dynamics $\dot{\tilde{\bm{x}}}(t)$ is proportional to the negative gradient $V^\prime(\tilde{\bm{x}}) = \frac{\partial V(\tilde{\bm{x}})}{\partial\tilde{\bm{x}}}$.

Several simulation results are shown in Fig. \ref{fig:C2andC3}, which were achieved considering the values in Table \ref{tab:valuesOfParam}. 
\begin{figure}
\begin{subfigure}{.245\textwidth}
  \centering
  \includegraphics[width=1\linewidth]{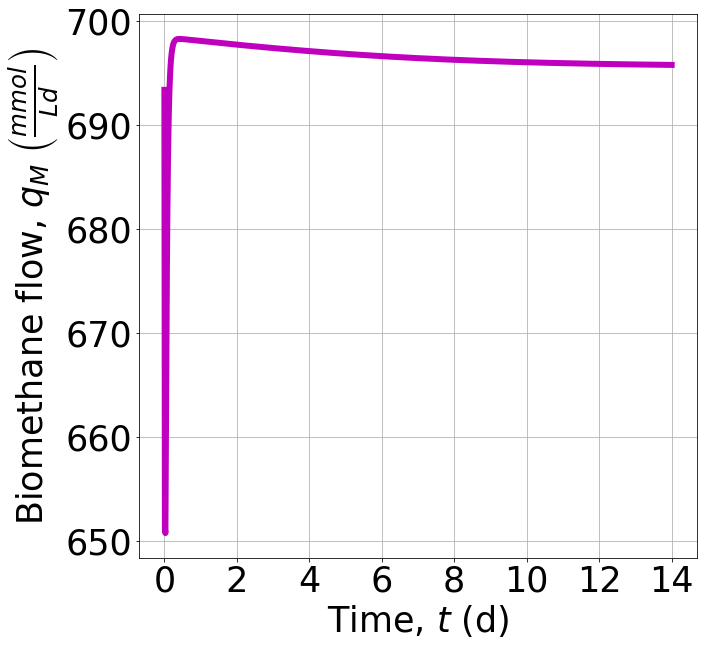}
  \caption{$P_d$, open loop}
  \label{fig:C2-OpenLoop}
\end{subfigure}%
\begin{subfigure}{.245\textwidth}
  \centering
  \includegraphics[width=1\linewidth]{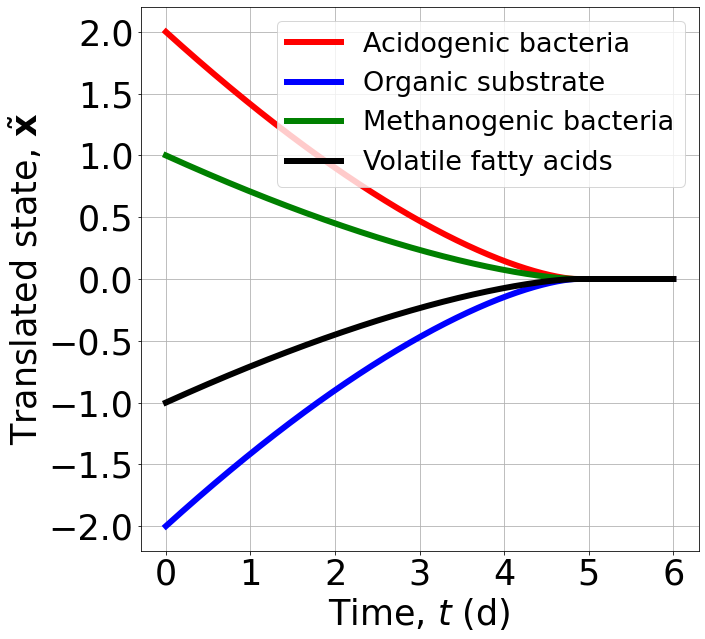}
  \caption{$P_d$, closed loop}
  \label{fig:C2-ClosedLoop1}
\end{subfigure}
\begin{subfigure}{0.24\textwidth}
  \centering
  \includegraphics[width=1\linewidth]{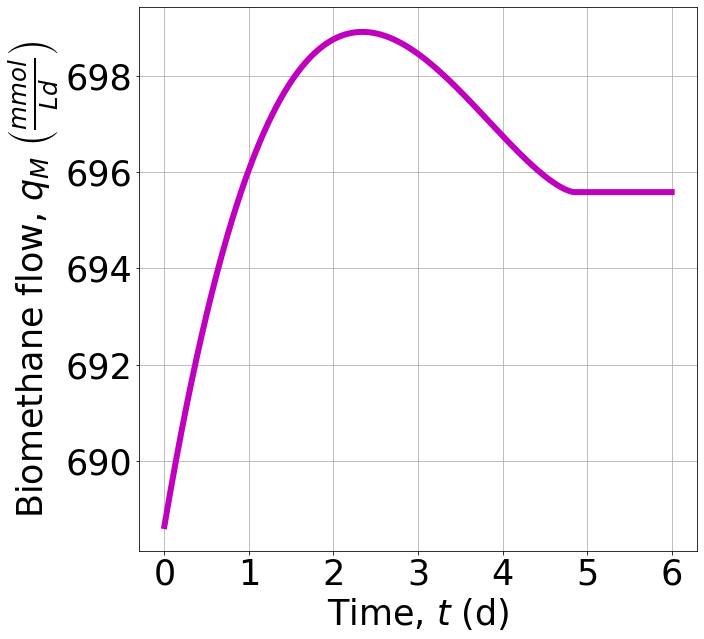}
  \caption{$P_d$, closed loop}
  \label{fig:C2-ClosedLoop2}
\end{subfigure}
\begin{subfigure}{.24\textwidth}
  \centering
  \includegraphics[width=1\linewidth]{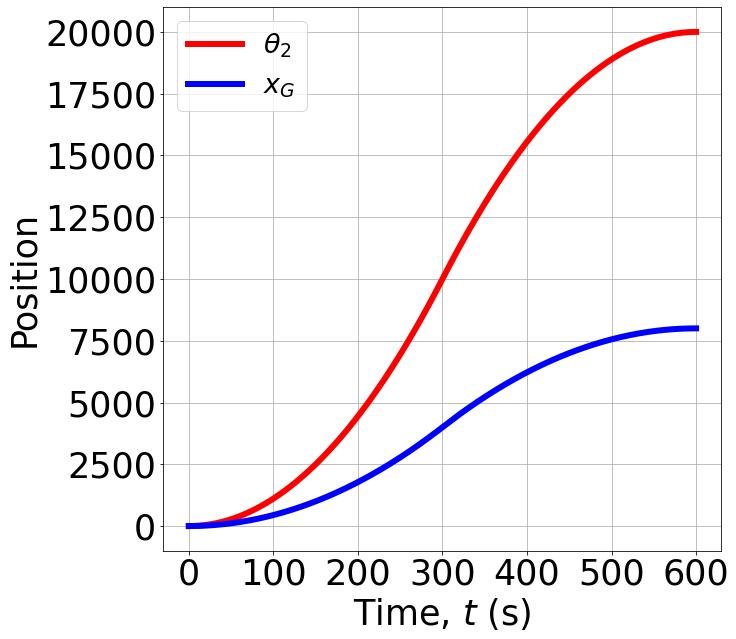}
  \caption{Truck}
  \label{fig:C3}
\end{subfigure}%
\caption{Simulation results of the plant digester $P_d$ and the truck.}
\label{fig:C2andC3}
\end{figure}      
\begin{table*}
\centering
\caption{Values of the parameters used for the numerical study.}
\label{tab:valuesOfParam}
\begin{tabular}{cccc} 
\hline
Compartment & Parameter & Value & Description\\ 
\hline
\hline
\multirow{3}{*}{$c^1_{1,1}$} & $m_l$ & 200 kg & Biomass truckload \\
& $m_{1,0}$ & 5000 kg & Initial hub stock\\
& $n_l$ & 7 & Loading time\\
\hline 
\multirow{11}{*}{$c^3_{1,2}$} & $m_l$ & 200 kg & Biomass truckload \\
& $m_v$ & 3500 kg & Mass of truck without load\\
& $I_z$ & 3000 kg$\text{m}^2$ & Yaw moment of inertia\\
& $a$ & 2 m & Chassis posterior length\\
& $b$ & 3 m & Chassis anterior length\\
& $r$ & 0.4 m & Radius of wheels \\
& $l$ & 2 m & Width of chassis\\
& $d$ & 0.1 m & Interaxis \\
& $H$ & 8000 m & Hub-plant distance \\
& $t_u$ & 600 s & Delivery time \\
& \makecell{$\theta_{1,0}$, $\dot{\theta}_{1,0}$, $\theta_{2,0}$, \\ $\dot{\theta}_{2,0}$, $\theta_{3,0}$, $\psi_{0}$} & 0 rad, 0 $\frac{\text{rad}}{\text{s}}$ & Initial conditions \\   
\hline
\multirow{7}{*}{$c^2_{2,2}$} & \makecell{$\alpha, S_{1\text{in}}$, $S_{2\text{in}}$, $k_1$, \\ $k_2$, $k_3$, $\mu_{1\text{max}}$, $\mu_{2\text{max}}$, \\ $K_{\text{S1}}$, $K_{\text{S2}}$, $K_{\text{I2}}$, $k_6$} & See \cite{bernard2001dynamical,campos2019hybrid} & See \cite{bernard2001dynamical,campos2019hybrid} \\
& $\bm{x}_0$ & $\left[3.43 \frac{\text{g}}{\text{L}}, -0.13 \frac{\text{g}}{\text{L}}, 7.19 \frac{\text{g}}{\text{L}}, 3.74 \frac{\text{mmol}}{\text{L}}\right]^\top$ & Initial conditions in open loop\\
& $\tilde{\bm{x}}_0$ & $\left[2 \frac{\text{g}}{\text{L}}, -2 \frac{\text{g}}{\text{L}}, 1 \frac{\text{g}}{\text{L}}, -1 \frac{\text{mmol}}{\text{L}}\right]^\top$ & Initial conditions in closed loop\\ [0.1cm]
 & $p$ & 1 & See \cite{haddad2015finite}\\ [0.1cm]
 & $\overline{D}$ & 0.5 $\frac{1}{\text{day}}$ & Dilution rate at equilibrium ``SS6''\\
\hline 
\end{tabular}
\end{table*}

\section{Conclusion}\label{sec:Conclusion}
The long-term unsustainability of the current take-make-dispose economy requires to redesign the material flows across natural and built environments. This paper establishes the foundations of a thermodynamics-based material flow design towards circularity by integrating the well-established design approach of thermodynamic cycles (e.g., Rankine cycle) with graph theory, dynamical systems, and control. Our examples demonstrate both the theoretical efficacy and the applicability of the proposed methodology. 

Future work will consist of designing TMNs by starting with few compartments and targeting materials whose circularity has high priority, e.g., atmospheric carbon dioxide and critical raw materials \cite{CRMs-EU,CRMs-US}. With respect to the existing literature, the goal is to complement MFA with dynamical power balances, graph theory, and control systems.

\bibliographystyle{IEEEtran}
\bibliography{mybibfile}

\vfill

\end{document}